\documentclass{amsart}
\usepackage{amsfonts,amssymb,amscd,amsmath,enumerate,verbatim,calc}

\newcommand{\CM}{Cohen-Macaulay}

\newcommand{\wrt}{with respect to}

\newcommand{\n}{\mathfrak{n} }
\newcommand{\m}{\mathfrak{m} }

\newcommand{\R}{\mathcal{R} }

\newcommand{\C}{\mathcal{C}}
\newcommand{\rt}{\rightarrow}

\newcommand{\ov}{\overline}

\theoremstyle{plain}

\newtheorem{theorem}{Theorem}[section]

\newtheorem{lemma}[theorem]{Lemma}

\theoremstyle{definition}
\newtheorem{definition}[theorem]{Definition}
\newtheorem{s}[theorem]{}
\newtheorem{remark}[theorem]{Remark}

\theoremstyle{remark}

\begin{document}

\title{On $p_g$-ideals}
\author{Tony~J.~Puthenpurakal}
\date{\today}
\address{Department of Mathematics, IIT Bombay, Powai, Mumbai 400 076}

\email{tputhen@math.iitb.ac.in}

 \begin{abstract}
Let $(A,\m)$ be an excellent normal domain of dimension two. We define an $\m$-primary ideal $I$ to be  a $p_g$-ideal if the Rees algebra $A[It]$ is a \CM \ normal domain.
When $A$ contains an algebraically closed field $k \cong A/\m$ then Okuma, Watanabe and  Yoshida proved that $A$ has $p_g$-ideals and furthermore product of two
$p_g$-ideals is a $p_g$ ideal. In this article we show that if $A$ is an   excellent normal domain of dimension two containing a  field $k \cong A/\m$ of characteristic zero
then also $A$ 
has $p_g$-ideals. Furthermore product of two $p_g$-ideals is $p_g$. 
\end{abstract}
 \maketitle
\section{introduction}
Zariski's theory of integrally closed ideals in a two dimensional regular local ring, $(A,\m)$, has been very influential; see \cite[Chapter 14]{HS} for a modern exposition. 
In particular product of two $\m$-primary integrally closed ideals is integrally closed. If the residue field of $A$ is infinite then every $\m$-primary integrally closed ideal $I$ is stable  i.e., for any minimal reduction $Q$ of $I$  we have $I^2 = QI$. In particular the Rees algebra $\R(I) = A[It]$ is a \CM \ normal domain (this also holds if $A/\m$ is finite).
Later 
Lipman proved that if $(A,\m)$ is a two dimensional rational singularity then  analogous results holds, see \cite{Lipman}.
However we cannot significantly weaken the hypotheses on $A$. In fact Cutkosky \cite{Cut} proved that if  $(A,\m)$ is an excellent normal local domain of dimension two such that $A/\m$ is algebraically closed and if for any $\m$-primary  integrally closed ideal $I$ we have $I^2$ is integrally closed then $A$ is a rational singularity.

Assume $(A,\m)$ is an excellent normal domain of dimension two  containing an algebraically closed field $k \cong A/\m$.
For such rings Okuma, Watanabe and  Yoshida in \cite{OWY-1} introduced (using geometric techniques) the notion of $p_g$-ideals. They showed that $p_g$ ideals are integrally closed.
If $I, J$ are two $\m$-primary $p_g$ ideals then $IJ$ is a $p_g$-ideal. Furthermore $I$ is stable and so  the Rees algebra $\R(I)$ is a \CM \ normal domain. 
They also proved that if $A$ is also a rational singularity then any $\m$-primary integrally closed ideal
is a $p_g$-ideal. In a later paper \cite{OWY-2} they showed that if $\R(I)$ is a \CM \ normal domain then $I$ is a $p_g$-ideal.
Motivated by this result we make the following definition:
\begin{definition}
 Let $(A.\m)$ be a normal domain of dimension two. An $\m$-primary ideal $I$ is said to be $p_g$-ideal in $A$ if the Rees algebra
 $\R(I) = A[It]$ is a normal \CM \  domain.
\end{definition}
We note that if $I$ is a $p_g$-ideal then all powers of $I$ are integrally closed. Furthermore if the residue field of $A$ is infinite then $I$ is stable; 
see \cite[Theorem 1]{GS}.  However from the definition
it does not follow that if $I, J$ are $p_g$-ideals then the product $IJ$ is also a $p_g$ ideal. Also we do not know that whether every normal domain of dimension two has a
$p_g$ ideal. In this paper we first prove:
\begin{theorem}\label{main-product}
Let $(A,\m)$ be an excellent two dimensional normal domain containing a perfect field $k \cong A/\m$. 
  If $I, J$ are $p_g$-ideals in $A$ then $IJ$ is also a $p_g$-ideal in $A$.
\end{theorem}
Regarding existence of $p_g$-ideals we prove
\begin{theorem}
 \label{main-existence}
 Let $(A,\m)$ be an excellent two dimensional normal domain containing a field $k \cong A/\m$ of characteristic zero.
 Then there exists $p_g$ ideals in $A$.
\end{theorem}
See Remark \ref{obs} to see the reason why our technique fails in positive characteristic.

We now describe in brief the contents of this paper. In section two we discuss
some preliminary results that we need. In section three we describe a construction and prove Theorem \ref{main-product}. In the next section we prove Theorem \ref{main-existence}.

\section{preliminaries}
In this section we prove the following preliminary result that we need. Parts of it are already known.

\begin{lemma}\label{basic}
Let $(A,\m)$ be a Noetherian local ring containing a perfect field $k \cong A/\m$. Let $\ell$ be a finite extension of $k$. Set $B = A \otimes_k \ell$. Then we have the following
\begin{enumerate}[\rm (1)]
 \item $B$ is a finite flat $A$-module.
 \item $B$ is a Noetherian ring.
 \item $B$ is local with  maximal ideal $\m B$ and residue field isomorphic to $\ell$.
 \item $B$ contains $\ell$.
 \item $A$ is \CM \ (Gorenstein, regular) \ if and only if $B$ is \CM \ (Gorenstein, regular).
 \item  If $A$ is excellent then so is $B$.
 \item If $A$ is normal then so is $B$.
 \item If $A$ is excellent normal and $I$ is an integrally closed ideal in $A$ then $IB$ is an integrally closed ideal in $B$.
 \item If $\ell$ is a Galois extension of $k$ with Galois group $G$ then $G$ acts on $B$ (via $\sigma(a\otimes t) = a\otimes \sigma(t)$). Furthermore if $|G|$ is invertible in $k$
 then $B^G = A$.
\end{enumerate}
\end{lemma}
\begin{proof}
As $k$ is perfect we have that $\ell$ is a separable extension of $k$. So by primitive element theorem we have that $\ell = k(\alpha)$. Let $f(x)$ be the minimal poynomial of 
$\alpha$. Then $B \cong A[X]/(f(x))$. We now prove our assertions.

(1)  This is clear.\\
(2) This follows from (1). \\
(3) Let $\n$ be a maximal ideal in $B$. Then as $B$ is finite over $A$ we get $\n \cap A = \m$. So $\n$ contains $\m B$. Notice
$B/\m B \cong k[x]/(f(x)) \cong \ell$. So $\m B$ is a maximal ideal in $B$. It follows $\n = \m B$. The result follows.\\
(4) This is clear. \\
(5) As $A$ is excellent so is $A[X]$. As $B$ is a quotient of $A[X]$ we get that $B$ is also excellent.\\
(6) The extension $A \rt B$ is flat with fiber $F \cong \ell$. 
The result follows from Corollary to Theorem 23.3, Theorem 23.4 and Theorem 23.7 in the text \cite{Mat}. 

(7) As $A$ is normal it satisfies $R_1$ and $S_2$. Let $P$ be a prime ideal in $P$ and let $\kappa(P)$ be the residue field of $A_P$. Then note that 
$B \otimes_A \kappa(P) = \ell\otimes_k \kappa(P)$ is a finite direct product of fields and so is regular. The result now follows from Theorem 23.9 in \cite{Mat}.

(8) As $A$ is normal then so is $A[t]$. Let $\R = A[It]$ be the Rees algebra of $A$ \wrt\ $I$. Set $\ov{\R} = \bigoplus_{n \geq 0}\ov{I^n}$ be the Rees-ring of the 
integral closure filtration of $I$. As $A$ is normal and excellent it follows that the completion $\widehat{A}$ is also normal. In particular it is reduced. So
So $\ov{\R}$ is a finite extension of $\R$. We note that $\ov{\R}$ is the integral closure of $\R$ in $A[t]$. So $\ov{\R}$ is normal. By graded version of (7)
we get that $\ov{\R}\otimes_k \ell$ is normal. We note that 
$\R\otimes_k \ell = \bigoplus_{n \geq 0}I^nB$ and $\ov{\R}\otimes_k \ell = \bigoplus_{n \geq 0}\ov{I^n}B$.
We have graded inclusions
\[
 \R\otimes_k \ell \subseteq \ov{\R}\otimes_k \ell \subseteq B[t].
\]
We note that as $B$ is normal we get $B[t]$ is normal. Also  $\ov{\R}\otimes_k \ell $ is a finite
 extension of $\R\otimes_k \ell$. It follows that $\ov{\R}\otimes_k \ell$ is the integral closure of $\R\otimes_k \ell$ in $B[t]$. In particular
 we have $\ov{I^n}B = \ov{I^nB}$ for $n \geq 1$. So for $n = 1$ we get $IB = \ov{I}B = \ov{IB}$. Thus $IB$ is integrally closed in $B$.

 (9) It is clear that $G$ acts on $B$ (via the action described) and $A \subseteq B^G$. Now assume $|G|$ is invertible in $k$. Let $\rho_k^\ell, \rho^B_A$ be the corresponding Reynolds operators.
Let $\xi = \sum_{i = 1}^{r}a_i\otimes t_i \in B^G$. Then note
\[
 \xi = \rho^B_A(\xi) = \sum_{i=1}^{r}\left( a_i\otimes \rho_k^\ell(t_i) \right) = \left(\sum_{i=1}^{r}a_i\rho^\ell_k(t_i) \right) \otimes 1  \in A.
\]

 \end{proof}

\section{A construction and proof of Theorem \ref{main-product}}
Throughout this section $(A,\m)$ is a Noetherian local ring containing a perfect field $k \cong A/\m$. \emph{Also throughout we assume $\dim A = 2$}. Fix an algebraic closure $\ov{k}$ of $k$. We investigate properties of 
$A\otimes_k \ov{k}$. Some of the results here are already known. However some of our applications regarding $p_g$-ideals is new and crucial
to prove Theorem \ref{main-product} and Theorem \ref{main-existence}.

\s \label{limit} Let 
$$\C_k = \{ E \mid E \ \text{is a finite extension of $k$ in $\ov{k}$} \}.$$
We note that  $\C_k$ is a directed system of fields with $\lim_{E \in \C_k} E  = \ov{k}$.
For $E \in \C_k$ set $A^E = A\otimes_k E$. Then by \ref{basic} $A^E$ is a finite flat extension of $A$. Also $A^E$ is local with maximal ideal $\m^E = \m A^E$.
Clearly $\{ A^E \}_{E \in \C_k}$ forms a directed system of local rings and we have $\lim_{E \in \C_k} A^E  = A\otimes_k \ov{k}$.
By \cite[Chap. 0. (10.3.13)]{EGA3} it follows that $A\otimes_k \ov{k}$ is a Noetherian local ring (say with maximal ideal $\m^{\ov{k}})$.
Note that we may consider $A^E$ as a subring of $A\otimes_k \ov{k}$. We have 
$$ A\otimes_k \ov{k} = \bigcup_{E\in \C_k}A^E \quad \text{and} \quad \m^{\ov{k}} = \bigcup_{E \in \C_k} \m^E.$$
It follows that $\m (A\otimes_k \ov{k}) = \m^{\ov{k}}$. It is also clear that $A\otimes_k \ov{k}$ contains $\ov{k}$ and its residue field
is isomorphic to $\ov{k}$. The extension $A \rt A\otimes_k \ov{k}$ is flat with fiber $\cong \ov{k}$. In particular $\dim A\otimes_k \ov{k}$ is two.

\s\label{cofinal} Let $F \in \C_k$. Set
\[
 \C_F = \{ E \mid E \in \C_k,  E\supseteq F \}.
\]
Then $\C_F$ is cofinal in $\C_k$. So we have $\lim_{E \in \C_F} A^E = A\otimes_k \ov{k}$.
Also note that if $E \in \C_F$ then 
\[
 A^E = A\otimes_k E = A\otimes_k F \otimes_F E = A^F \otimes_F E.
\]
It also follows that $\m^E = \m^FA^E$.

The following result is definitely known to experts. We give a proof for the convenience of the reader.
\begin{lemma}
 \label{excellent}
 If $A$ is excellent then so is $A\otimes_k \ov{k}$
\end{lemma}
\begin{proof}
 In the directed system $\{ A^E \}_{E \in \C_k}$ each map $A^F \rt A^E$ (when $F\subseteq E$) is etale as $A^E = A^F\otimes_F E$ and $E$ is seperable over $F$.
 So by a result of \cite[5.3]{G} it follows that $A\otimes_k \ov{k}$ is excellent.
\end{proof}

We now show the main properties of $A\otimes_k \ov{k}$ that we need
\begin{theorem}
 \label{eclair}
 (with hypotheses as above) Set $T = A\otimes_k {\ov{k}}$ and $\n = \m^{\ov{k}}$. We have
 \begin{enumerate}[\rm (1)]
  \item  $A$ is \CM \ (Gorenstein, regular) if and only if $T$ is \CM \ (Gorenstein, regular).
  \item If $A$ is normal domain if and only if $T$ is a normal domain.
  \item Assume $A$ is also excellent normal domain. Then we have
  \begin{enumerate}[\rm (a)]
   \item $I$ is integrally closed in $A$ if and only if $IT$ is integrally closed in $T$
   \item $I$ is a $p_g$ ideal in $A$ if and only if $IT$ is a $p_g$ ideal in $T$.
  \end{enumerate}
 \end{enumerate}
\end{theorem}
\begin{proof}
(1) The extension $A \rt T$ is flat local with fiber ring $\ov{k}$. The result follows from Corollary to Theorem 23.3, Theorem 23.4 and Theorem 23.7 in the
text \cite{Mat}.

(2) If $A$ is normal then so is $A^E$ for every $E \in \C_k$. In particular $T = \bigcup_{E \in \C_k}A^E$ is a domain.
If $R$ is a domain let $K(R)$ denote the fraction field of $R$. Clearly $R(T) = \bigcup_{E \in \C_k}R(A^E)$.
Let $\xi = a/b \in R(T)$ be integral over $T$.
Then $\xi$ satisfies a monic polynomial $h(x) \in T[x]$. Choose $E \in \C_k$ such that $a,b$ and all coefficients of $h$ are in $A^E$. Then
$\xi \in R(A^E)$ is integral over $A^E$. As $A^E$ is normal we have $\xi \in A^E $. So $\xi \in T$. Thus $T$ is normal.
Conversely if $T$ is normal then as the extension $A \rt T$ is flat we get by Corollary to Theorem 23.7 in \cite{Mat} we get that $A$ is normal.

(3)(a) If $IT$ is integrally closed in $T$ then $IT\cap A$ is integrally closed in $A$. But $A \rt T$ is faithfully flat. So $IT \cap A = I$ (note we did not use excellence of
$A$ to prove this). Conversely assume $I$ is integrally closed in $A$. As $A$ is excellent and normal, by Lemma \ref{basic}(8),
we have that $IA^E$ is integrally closed in $A^E$ for every $E \in \C_k$. Let
$\xi \in T$ be integral over $IT$. Say
we have an equation
\[
 \xi^n + a_1\xi^{n-1} + \cdots a_{n-1}\xi + a_n = 0
\]
with $a_i \in (IT)^i = I^iT$. We may choose $F \in \C_k$ such that $\xi \in A^T$ and $a_i \in I^iA^F$. So $\xi$ is integral over $IA^F$. But $IA^F$ is integrally
closed. Therefore $\xi \in IA^F$. So $\xi \in IT$. Thus $IT$ is integrally closed.

(3)(b) Let $\R(I), \R(IT)$ be the Rees Algebra of $I$ and $IT$ respectively. Notice $\R(IT) = \R(I)\otimes_A T = \R(I)\otimes_k \ov{k}$.
The rings $\R(I)$ and $\R(IT)$ are $*$-local. Furthermore the extension $\R(I) \rt \R(IT)$ is flat with fiber $\ov{k}$. So by graded analog of (1)  we get that
$\R(I)$ is \CM \ if and only if $\R(IT)$ is \CM.

First assume $I$ is a $p_g$ ideal in $A$. Then $I^n$ is integrally closed in $A$ for all $n \geq 1$. By 3(a) we get that $I^nT$ is integrally closed in $T$ for all $n \geq 1$.
Also as $T$ is normal we get $T[t]$ is normal. As $\R(IT)$ is integrally closed in $T[t]$ we get that it is a normal domain. Also as $\R(I)$ is \CM,  as discussed earlier
we get that $\R(IT)$ is \CM. So $IT$ is a $p_g$ ideal in $T$.

Conversely assume that $IT$ is a $p_g$ ideal in $T$. Then $(IT)^n = I^nT$ is integrally closed for all $n \geq 1$. By 3(a) we get that $I^n$ is integrally closed for all 
$n \geq 1$. As $A$ is normal, as argued before we get that $\R(I)$ is normal. Also as $\R(IT)$ is \CM, as discussed earlier we get that $\R(I)$ is normal. So
$I$ is a $p_g$ ideal in $A$.
\end{proof}

We now give
\begin{proof}[Proof of Theorem \ref{main-product}]
 Set $T = A\otimes_k \ov{k}$. Let $\n$ be the maximal ideal of $T$
 We note that $T$ is an excellent normal domain containing $\ov{k} \cong T/\n$ (see \ref{limit}, \ref{excellent} and \ref{eclair}(2)). 
 Let $I, J$ be two $p_g$ ideals in $A$. Then by \ref{eclair}3(b) we get that $IT, JT$ are $p_g$ ideals in $T$. By \cite[3.5]{OWY-1} we get that
 $(IT)(JT) = IJT$ is a $p_g$ ideal in $T$. So again by \ref{eclair}3(b) we get that $IJ$ is a $p_g$-ideal in $A$.
\end{proof}

\section{proof of Theorem \ref{main-existence}}
In this section we give 
\begin{proof}[Proof of Theorem \ref{main-existence}]Set $T = A\otimes_k \ov{k}$. Let $\n$ be the maximal ideal of $T$
 We note that $T$ is an excellent normal domain containing $\ov{k} \cong T/\n$ (see \ref{limit}, \ref{excellent} and \ref{eclair}(2)). 
 By \cite[4.1]{OWY-1} there exists a $p_g$ ideal $J$ in $T$. By \ref{limit} we have $T = \bigcup_{E \in \C_k}A^E$. So there exists $F \in \C_k$ which contains a set of minimal generators 
 of $J$. We may further assume  (by enlarging) that $F$ is Galois over $k$. Thus there exists ideal $W$ in $A^F$ with $WT = J$. By \ref{eclair}(3)(b) we get that
 $W$ is a $p_g$ ideal in $A^F$. Let $G$ be the Galois group of $F$ over $k$. Then $G$ acts on $A^F$ (via $\sigma(a\otimes f) = a\otimes \sigma(f)$). As $k$ has characteristic
 zero we have by \ref{basic}(9) that $(A^F)^G = A$. We also note that we have a natural $G$ action on $A^F[t]$ (fixing $t$) and clearly  its invariant ring is $A[t]$.
 Let $\sigma \in G$. It's action on $A^F[t]$ induces an isomorphism of between the Rees algebra's $\R(W)$ and $\R(\sigma(W))$. So $\sigma(W)$ is a $p_g$ ideal 
 in $A^F$.
 By Theorem \ref{main-product} we get that $K = \prod_{\sigma \in G}\sigma(W)$ is a $p_g$ ideal in $A^F$. Note $K$ is $G$-invariant. So the $G$ action of $A^F[t]$
 restricts to a $G$-action on $\R(K)$. As characteristic $k$ is zero we get that $V = \R(K)^G$ is a \CM \ normal subring of $A[t]$.
 Set $V = \bigoplus_{n \geq 0}V_n$. Note $V_0 = A$ and $V_n = K^n \cap A$ are integrally closed $\m$-primary ideals of $A$. Note $V$ is not necessarily standard graded. However
 it is well-known that a Veronese subring $V^{<l>} $ of $V$ is standard graded. Note $V^{<l>}$ is a \CM \ normal domain. Observe that
 $V^{<l>} = \R(V_l)$. Thus $V_l$ is a $p_g$-ideal in 
 $A$.
\end{proof}

\begin{remark}\label{obs}
 Our proof of Theorem \ref{main-existence} would go through in positive characteristic would go through if we knew order of $G$ is invertible in $k$. However 
 we have no control on $G$. So our proof does not extend in this case.
\end{remark}

\section*{Acknowledgements}
I thank Keiichi Watanabe and Ken-ichi Yoshida for many fruitful discussions regarding this paper.


\begin{thebibliography}{10}
\bibitem{Cut}
S.~D.~Cutkosky,
\emph{A new characterization of rational surface singularities},
Invent. Math. 102 (1990), 157–-177. 

\bibitem{GS}
S.~Goto and Y.~Shimoda,
\emph{On the Rees Algebra of Cohen-Macaulay local rings},
Commutative algebra (Fairfax, Va., 1979), 
pp 201-231, Lecture Notes in Pure and Appl. Math., 68, Dekker, New York.

\bibitem{G}
S.~Greco,
\emph{Two theorems on excellent rings},
Nagoya Math. J.Vol. 60 (1976), 139--149

\bibitem{EGA3}
A.~Grothendieck and J.~A.~Dieudonn\'{e},
\emph{\'{E}l\'{e}ments de g\'{e}ome\'{e}trie alg\'{e}brique}.
Chap III (part1), Inst. Etudes Sci. Publ. Math. 24, 1965



\bibitem{HS}
C.~Huneke and I.~Swanson,
\emph{Integral closure of ideals, rings, and modules},
London Mathematical Society Lecture Note Series, 336. Cambridge University Press, Cambridge, 2006. 


\bibitem{Lipman}
J.~Lipman, 
\emph{Rational singularities with applications to algebraic surfaces and unique factorization},
Inst. Hautes Études Sci. Publ. Math. 36 (1969), 195–-279.

\bibitem{Mat}
H. Matsumura,
\emph{Commutative ring theory}, second ed., 
Cambridge Studies in Advanced Mathematics, vol. 8, Cambridge University Press, Cambridge, 1989

\bibitem{OWY-1}
T.~Okuma, K.~Watanabe, and K.~Yoshida, Ken-ichi
\emph{Good ideals and pg-ideals in two-dimensional normal singularities},
Manuscripta Math. 150 (2016), no. 3-4, 499–-520.

\bibitem{OWY-2}
\bysame,
\emph{Rees algebras and pg-ideals in a two-dimensional normal local domain}, 
Proc. Amer. Math. Soc. 145 (2017), no. 1, 39–-47.

\end{thebibliography}
\end{document}